\documentclass{article}

\usepackage{authblk}
\title{Infinite Communication Complexity and KW Games}
\author{Evan Leach}
\affil{Department of Mathematics, UCLA}
\date{August 31, 2026}

\usepackage[letterpaper, total={5.5in, 8in}]{geometry}
\usepackage{setspace}
\usepackage{amsmath}
\usepackage{amssymb}
\usepackage{amsthm}
\usepackage{mathrsfs}
\usepackage{enumitem}
\usepackage{forest}
\usepackage{bm}
\usepackage{hyperref}
\hypersetup{
    colorlinks=true,
    linkcolor=blue,
    filecolor=blue,      
    urlcolor=blue,
    citecolor=magenta,
    }

\usepackage[
    backend=biber,
    style=alphabetic,
    maxalphanames=3,
    maxbibnames=99
]{biblatex}

\newcommand{\NN}{\mathbb{N}}

\newcommand{\AND}{\wedge}
\newcommand{\OR}{\vee}
\newcommand{\NOT}{\neg}
\newcommand{\BSigma}{\boldsymbol{\Sigma}}
\newcommand{\BPi}{\boldsymbol{\Pi}}
\newcommand{\NC}{\mathsf{NC}^1}

\newcommand{\Pclass}{\mathsf{P}}
\newcommand{\NP}{\mathsf{NP}}
\newcommand{\coNP}{\mathsf{coNP}}
\newcommand{\poly}{\mathsf{poly}}

\newcommand{\EQ}{\mathrm{EQ}}
\newcommand{\DISJ}{\mathrm{DISJ}}
\newcommand{\KW}{\mathrm{KW}}
\newcommand{\mKW}{\mathrm{mKW}}
\newcommand{\BCG}{\mathrm{BCG}}
\newcommand{\mBCG}{\mathrm{mBCG}}
\DeclareMathOperator{\rank}{rank}

\theoremstyle{plain}
\newtheorem{theorem}{Theorem}[section]

\newtheorem{lemma}[theorem]{Lemma}
\newtheorem{corollary}[theorem]{Corollary}

\theoremstyle{definition}
\newtheorem{definition}[theorem]{Definition}

\newtheorem{remark}[theorem]{Remark}

\begin{document}

\null
\nointerlineskip
\vfill
\let\snewpage \newpage
\let\newpage \relax
\maketitle

\vspace{2cm}

\begin{abstract}

We characterize the Borel sets with an infinite version of Karchmer and Wigderson's game linking finite circuit complexity to communication complexity. To this end, we formulate an infinite version of communication complexity and prove that a given subset of the Cantor space is Borel if and only if a certain infinite communication game is solvable. We utilize this connection to provide new elementary and purely combinatorial proofs of some classical results in descriptive set theory, including the analytic separation theorem and the equivalence of monotone and positive Borel sets. Another consequence is a characterization of Borel separability via the winner of a certain ``cut-and-choose'' game, which we use to obtain new combinatorial proofs that neither the set of ill-founded trees nor any infinite parity function is Borel.

\end{abstract}

\vspace{2cm}

\let \newpage \snewpage
\vfill

\newpage

\tableofcontents

\newpage

\section{Introduction}

Karchmer--Wigderson (KW) games \cite{kw1990} characterize the minimum circuit depth needed to compute a boolean function $f \colon \{0,1\}^n \rightarrow \{0,1\}$ in terms of a communication game. This connection allows us to prove elusive circuit lower bounds by transferring the problem to the field of communication complexity. KW games have been used to prove several state-of-the-art lower bounds, such as the monotone circuit depth of perfect matching \cite{raz1992}. They have also seen modern successes in circuit complexity \cite{multiparty2022} and remain a promising technique for resolving major open problems such as the $\NC$ vs.\ $\Pclass$ question \cite{krw1995, krw2020}.

In this paper, we formulate an infinite version of communication complexity and generalize KW games to the infinite setting. In particular, we extend Karchmer and Wigderson's results to characterize sets computable by infinite circuits in terms of infinite KW games.

Infinite circuits have been studied by Sipser in \cite{sipser1983} and \cite{sipser1984}. These circuits are countable, well-founded trees with nodes labeled as $\AND$-gates, $\OR$-gates, literals $x_i$, and negated literals $\overline{x}_i$. Infinite circuits compute maps which take a string $x \in 2^\NN$ as input and output the evaluation of the root node to $0$ or $1$. The functions $f \colon 2^\NN \rightarrow \{0,1\}$ computed by infinite circuits are precisely the characteristic functions of Borel sets $B \subseteq 2^\NN$ (see Section~\ref{sec:correspondence} for a proof). We can also characterize analytic sets in terms of nondeterministic circuits.

\begin{figure}[h]
    \centering
    \begin{forest}
      for tree={
        math content,
        s sep=3mm, 
        l sep=12mm, 
        tier/.option=level,
        edge={thick}
      },
      gate/.style={circle, draw, thick, fill=gray!20, inner sep=0pt, minimum size=6mm},
      ellipsis child/.style={
        edge path={
          \noexpand\draw[\forestoption{edge}] (!u) -- ($()!4.5mm!(!u)$);
          \noexpand\fill ($()!3mm!(!u)$) circle (0.7pt);
          \noexpand\fill ($()!1.5mm!(!u)$) circle (0.7pt);
          \noexpand\fill () circle (0.7pt);
        }
      }
      [\bm{\OR}, gate
        [\bm{\AND}, gate
          [x_0]
          [\overline{x}_1]
          [\overline{x}_2]
          [, ellipsis child]
          [, ellipsis child]
        ]
        [\bm{\AND}, gate
          [x_1]
          [\overline{x}_2]
          [\overline{x}_3]
          [, ellipsis child]
          [, ellipsis child]
        ]
        [\bm{\AND}, gate
          [x_2]
          [\overline{x}_3]
          [\overline{x}_4]
          [, ellipsis child]
          [, ellipsis child]
        ]
        [, coordinate, no edge]
        [, coordinate, no edge]
        [, coordinate, no edge]
        [, coordinate, no edge]
        [, coordinate, no edge]
        [, ellipsis child]
        [, coordinate, no edge]
        [, coordinate, no edge]
        [, coordinate, no edge]
        [, coordinate, no edge]
        [, ellipsis child]
      ]
    \end{forest}
    \caption{An example of an infinite circuit.}
    \label{fig:circuit}
\end{figure}
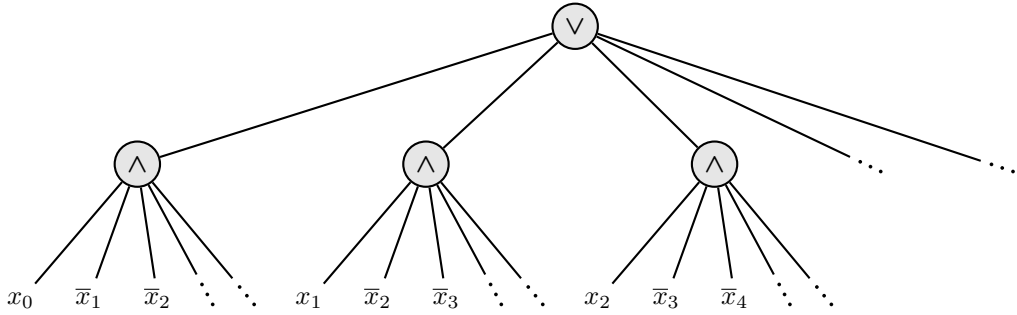

Our generalization of communication complexity to the infinite setting is new. In our model, Alice and Bob are tasked with computing a function $f \colon X \times Y \rightarrow Z$, where $X = Y = 2^\NN$ and $Z$ is countable. Alice has access to the input $x \in X$, while Bob has access to the input $y \in Y$. They can communicate with each other by sending natural numbers, and they must eventually output $f(x,y)$. We can view such a procedure as a countable decision tree, where internal nodes contain functions from $X$ or $Y$ to $\NN$ telling us which child to progress to when traversing down the tree (we interpret this output as the natural number Alice or Bob chooses to communicate). Leaf nodes are labeled by elements of $Z$ representing the procedure's output. If such a tree is well-founded, we call it a \emph{protocol}. We show that the \emph{fooling set} method from finite communication complexity generalizes seamlessly to this new framework and use it to prove that the infinite analogs of the equality and set disjointness problems do not admit communication protocols.

One communication problem of particular interest is the \emph{infinite KW game} $\KW_{A,B}$, where $A$ and $B$ are subsets of $2^\NN$. Alice is given a string $x \in A$, Bob is given a string $y \in B$, and they must output an index $i \in \NN$ such that $x_i \neq y_i$. We prove the following infinite version of Karchmer and Wigderson's theorem in \cite{kw1990}:

\begin{theorem}
\label{borel-kw-equivalence}

Two sets $A,B \subseteq 2^\NN$ are Borel separable if and only if $\KW_{A,B}$ has a communication protocol.

\end{theorem}

Furthermore, the minimal rank of a Borel set separating $A$ and $B$ is closely related to the minimal ordinal rank of the underlying tree for a protocol solving the game $\KW_{A,B}$. Both directions of Theorem~\ref{borel-kw-equivalence} are useful. The reverse direction gives us several existence results, including a completely elementary and combinatorial proof of the analytic separation theorem (Theorem~\ref{analytic-separation}). We also prove that a monotone function is computed by a circuit if and only if it is computed by a monotone circuit (Theorem~\ref{monotone-borel}). The analogous claim is false for polynomial-sized circuit families \cite{tardos1988}, so this result is an interesting demonstration of differences between finite and infinite circuits.

The forward direction of Theorem~\ref{borel-kw-equivalence} seems to be a useful tool for impossibility results, since one can show that two sets $A,B \subseteq 2^\NN$ are not Borel separable by establishing hardness results for $\KW_{A,B}$. Unfortunately, there is an obstruction (Theorem~\ref{kw-partition}) telling us that most techniques from finite communication complexity cannot prove hardness results for infinite KW games.

To circumvent this obstruction, we formulate a new \emph{cut-and-choose game} $\BCG_{A,B}$, where $A$ and $B$ are subsets of $2^\NN$. This \emph{Borel cut-and-choose game} is a competitive game between Player I and Player II. Player I partitions either $A$ or $B$ into countably many pieces, Player II chooses a piece to keep, and the remaining pieces are discarded. Player I wins if and only if there exists some finite stage at which we have for some fixed $i \in \NN$ that $x_i \neq y_i$ for all $(x,y) \in A \times B$. We prove that $A$ and $B$ are Borel separable if and only if Player I has a winning strategy in $\BCG_{A,B}$. The forward direction of this result lets us show that two sets $A,B \subseteq 2^\NN$ are not Borel separable by providing a winning strategy for Player II in $\BCG_{A,B}$. This strategy can be viewed as a procedure which finds for any protocol $P$ some leaf node where the protocol must output an incorrect value. In other words, the Borel cut-and-choose game formalizes a technique for proving Borel inseparability by showing that an appropriate KW game does not admit a communication protocol.

We use this technique to prove that no infinite parity function is Borel (Corollary~\ref{parity}) and that the set of ill-founded trees is not Borel (Theorem~\ref{ill-founded-trees}). The proof of Theorem~\ref{ill-founded-trees} is quite different from the standard argument using universal sets and diagonalization, and there is a precise way in which it gives a simpler winning strategy for Player II in the appropriate Borel cut-and-choose game than one would expect.

Though the applications of KW games we discuss are not new results, the proofs are purely combinatorial and elementary. In particular, we do not rely on any external machinery from descriptive set theory. Many of the results we prove have major open problems regarding $\Pclass$, $\NP$, and $\coNP$ as finite analogs. By finding combinatorial proofs of the infinite versions, we can hopefully develop ideas helpful in eventually resolving these questions.

\subsection{Outline of the paper}

We discuss preliminaries in Section~\ref{sec:preliminaries}, covering basic definitions from descriptive set theory and providing an overview of finite communication complexity. We then discuss infinite circuits and their connection to Borel sets in Section~\ref{sec:infinite-circuits}. We show that a subset of $2^\NN$ is Borel if and only if it is computable by a circuit (Theorem~\ref{circuit-borel}) and is analytic if and only if it is computable by a depth-$2$ nondeterministic circuit (Theorem~\ref{depth-2}).

Our new contributions begin in Section~\ref{sec:infinite-cc}, where we introduce and motivate our model of infinite communication complexity. In Section~\ref{sec:kw}, we formulate our infinite analog of Karchmer--Wigderson games. We prove Theorem~\ref{borel-kw-equivalence}, as well as a variant for Borel sets computable by monotone circuits (i.e.\ the \emph{positive} Borel sets, see \cite{kechris1995}).

In Section~\ref{sec:existence}, we present proofs using KW games of the analytic separation theorem (Theorem~\ref{analytic-separation}) and the fact that a Borel set is monotone if and only if it is computable by a monotone circuit (Theorem~\ref{monotone-borel}). In Section~\ref{sec:impossibility}, we discuss impossibility results using KW games. We first describe the obstruction preventing us from using standard techniques (namely those which find lower bounds for partition numbers), and we then prove Theorem~\ref{bcg} and a monotone variant. We conclude by proving the aforementioned hardness results for parity and ill-founded trees. Finally, we discuss the difficulties of translating these arguments to the finite setting in Section~\ref{sec:finite-circuits}.

\section{Preliminaries}
\label{sec:preliminaries}

\subsection{Borel and analytic sets}

We begin with some standard definitions from descriptive set theory. See \cite{kechris1995} for a detailed treatment. Define $2^{<\NN} = \bigcup_{n \in \NN} \{0,1\}^n$, and let $w_0w_1 \cdots w_{k-1}w_k$ denote the concatenation of strings $w_0,w_1, \dots, w_{k-1} \in 2^{<\NN}$ and $w_k \in 2^{<\NN} \cup 2^\NN$. We consider the \emph{Cantor space} $2^\NN$, equipped with the topology generated by basic open sets of the form $w2^\NN$ for $w \in 2^{<\NN}$.

\begin{definition}

A set $S \subseteq 2^\NN$ is called $\BSigma^0_0$ or $\BPi^0_0$ if it is clopen. Given $1 \leq \alpha < \omega_1$, we call a set $\BSigma^0_\alpha$ (resp.\ $\BPi^0_\alpha$) if it is a countable union (resp.\ intersection) of sets $S_n$ for $n \in \NN$, each of which is in $\BPi^0_{\beta_n}$ (resp.\ $\BSigma^0_{\beta_n}$) for some $\beta_n < \alpha$.

\end{definition}

This definition agrees with the usual definition of the Borel hierarchy since the Cantor space is zero-dimensional (has a clopen basis). In particular, the open (resp.\ closed) sets of $2^\NN$ are precisely the countable unions (resp.\ intersections) of clopen sets.

\begin{definition}

A set $B \subseteq 2^\NN$ is \emph{Borel} if it is $\BSigma^0_\alpha$ for some $\alpha < \omega_1$. The \emph{rank} of a Borel set $B$ is the smallest $\alpha$ such that $B \in \BSigma^0_\alpha \cup \BPi^0_\alpha$.

\end{definition}

\begin{definition}

A set $A \subseteq 2^\NN$ is \emph{analytic} if it is the projection of a Borel set; i.e.\ if there exists a Borel set $B \subseteq 2^\NN \times 2^\NN \cong 2^\NN$ such that
\begin{equation*}
    A = \{x \in 2^\NN : (x,y) \in B \text{ for some } y \in 2^\NN\} \text{.}
\end{equation*}

\end{definition}

\subsection{Finite communication complexity}
\label{sec:finite-cc}

The field of \emph{communication complexity} studies the following problem: Suppose Alice is given a string $x \in X$ and Bob is given a string $y \in Y$, where $X$ and $Y$ are finite sets. How many bits must they communicate in order to compute a given function $f \colon X \times Y \rightarrow Z$? This section contains a very brief summary of the field, including only what is needed for our purposes. Kushilevitz and Nisan give an excellent introduction to the subject in \cite{nisan97}.

Strategies for communication problems are formalized using the concept of a \emph{protocol}, which is a finite binary tree with internal nodes belonging to either Alice or Bob. Internal nodes $v$ are labeled either with functions $a_v \colon X \rightarrow \{0,1\}$ (for Alice's nodes) or $b_v \colon Y \rightarrow \{0,1\}$ (for Bob's nodes) describing what Alice or Bob chooses to communicate at that step. These functions can be arbitrary, reflecting the fact that Alice and Bob have unbounded computational power. Leaf nodes are labeled with elements of $Z$, denoting outputs of the protocol.

A protocol $P$ computes a function $P \colon X \times Y \rightarrow Z$, where one starts at the root node and traverses down the tree, computing $a_v(x)$ or $b_v(y)$ at each internal node $v$ to determine which child to progress to (the left child if $a_v$ or $b_v$ evaluates to $0$, and the right child otherwise), and outputting the value of the leaf node eventually reached. The \emph{cost} of a protocol is the height of the tree. These protocol trees correspond precisely to strategies for Alice and Bob, with the cost of the protocol corresponding to the number of bits communicated in the worst case.

\begin{definition}

The \emph{communication complexity} of a function $f \colon X \times Y \rightarrow Z$, denoted $D(f)$, is the minimum cost of a protocol computing $f$.

\end{definition}

\subsubsection{Fooling sets and lower bounds}

The primary goal of communication complexity is to prove lower bounds on $D(f)$. Theorem~\ref{finite-partition-bound}, proved by Yao in his seminal paper on communication complexity \cite{yao1979}, is the standard technique for this purpose.

\begin{definition}

A \emph{rectangle} in $X \times Y$ is a set of the form $A \times B$, with $A \subseteq X$ and $B \subseteq Y$. We call a rectangle $R$ \emph{$f$-monochromatic} if $f$ is constant on $R$. The \emph{partition number} $\chi(f)$ of a function $f$ is the smallest number of $f$-monochromatic rectangles needed to partition $X \times Y$.

\end{definition}

\begin{theorem}
\label{finite-partition-bound}

We have $D(f) \geq \lceil \log_2 \chi(f) \rceil$.

\end{theorem}

\begin{proof}

Suppose $f$ has a protocol $P$ of cost $c$. For each leaf node $v$ of $P$, there are functions $a_0,\dots,a_k,b_0,\dots,b_\ell$ and bits $q_0,\dots, q_k,r_0,\dots,r_\ell$ such that the protocol reaches $v$ on input $(x,y) \in X \times Y$ if and only if $a_i(x) = q_i$ for each $0 \leq i \leq k$ and $b_j(y) = r_j$ for each $0 \leq j \leq \ell$. Thus $P$ is constant on the rectangle
\begin{equation*}
    R_v = \left(\bigcap_{i=0}^k a_i^{-1}(\{q_i\})\right) \times \left(\bigcap_{j=0}^\ell b_j^{-1}(\{r_j\})\right) \subseteq X \times Y \text{.}
\end{equation*}
The rectangles $R_v$ for all leaf nodes $v$ of $P$ partition $X \times Y$, and since $P$ computes $f$, these rectangles are $f$-monochromatic. There are at most $2^c$ such rectangles, and this implies that $2^{D(f)} \geq \chi(f)$. The claim follows by taking logarithms.
\end{proof}

This theorem allows us to establish lower bounds for $D(f)$ by finding lower bounds for $\chi(f)$. One of the most basic techniques for this purpose is called the \emph{fooling set method}.

\begin{definition}

We call a set $\mathcal{F} \subseteq f^{-1}(\{1\})$ a \emph{fooling set} for $f$ if for any distinct elements $(x,y),(x',y') \in \mathcal{F}$, we have that $f(x,y')$ or $f(x',y)$ is equal to $0$.

\end{definition}

\begin{lemma}
\label{fooling-sets}

If $\mathcal{F}$ is a fooling set for $f$ and $f^{-1}(\{0\}) \neq \emptyset$, then $\chi(f) \geq |\mathcal{F}| + 1$.

\end{lemma}

\begin{proof}

Any $f$-monochromatic rectangle can contain at most one element of $\mathcal{F}$, so at least $|\mathcal{F}|$ $f$-monochromatic rectangles are required to cover $f^{-1}(\{1\})$. An additional rectangle is needed to cover $f^{-1}(\{0\})$ since this set is nonempty.
\end{proof}

We now demonstrate the fooling set technique by proving lower bounds on the communication complexity of two explicit functions:

\begin{definition}

We define functions $\EQ_n, \DISJ_n \colon \{0,1\}^n \times \{0,1\}^n \rightarrow \{0,1\}$ for each $n \in \NN$ as follows: We set $\EQ_n(x,y) = 1$ if $x = y$ and $\EQ_n(x,y) = 0$ otherwise. Similarly, we set $\DISJ_n(x,y) = 1$ if and only if $x_i \wedge y_i = 0$ for all $0 \leq i \leq n-1$.

\end{definition}

\begin{corollary}

We have $D(\EQ_n) = D(\DISJ_n) = n+1$.

\end{corollary}

\begin{proof}

The upper bound follows from the fact that Alice can use $n$ bits to share $x$ with Bob, and Bob can then compute $\EQ_n(x,y)$ or $\DISJ_n(x,y)$ himself and use a single bit of communication to branch to the correct leaf node. For the lower bound, notice that $\{(x,x) : x \in \{0,1\}^n\}$ and $\{(x,\overline{x}) : x \in \{0,1\}^n\}$ are fooling sets for $\EQ_n$ and $\DISJ_n$, respectively (where $\overline{x}_i$ denotes the negation of $x_i$ and $\overline{x} = \overline{x}_0 \cdots \overline{x}_{n-1}$). Thus Lemma~\ref{fooling-sets} tells us that $\chi(\EQ_n),\chi(\DISJ_n) \geq 2^n + 1$, and Theorem~\ref{finite-partition-bound} gives us a lower bound of $\lceil \log_2(2^n + 1)\rceil = n + 1$ for $D(\EQ_n)$ and $D(\DISJ_n)$.
\end{proof}

\section{Infinite circuits}
\label{sec:infinite-circuits}

\begin{definition}

A \emph{circuit} $C$ on variables $x_0, x_1, \dots$ is a countable, labeled, well-founded tree whose nodes are called \emph{gates} and are labeled according to the following conditions:
\begin{itemize}
    \item Any node of the tree with at least $1$ child is called an \emph{internal gate} and must be labeled with either the symbol $\AND$ or $\OR$.
    \item The leaf nodes are called \emph{input gates} and must be labeled with either $0$, $1$, or the symbol $x_i$ or $\overline{x}_i$ for some $i \in \NN$ (denoting either the $i$-th bit of $x$ or the negation of this bit). We call the labels $0,1,x_i,\overline{x}_i$ \emph{literals}.
\end{itemize}

\end{definition}

\begin{definition}

Given a gate $G$ in $C$ and some input $x \in 2^\NN$, we define the \emph{evaluation} $G(x)$ inductively (which we may do by well-foundedness). An input gate evaluates to its label, an $\AND$-gate evaluates to the conjunction of the evaluations of its children, and an $\OR$-gate evaluates to the disjunction of the evaluations of its children. The root node $G_0$ of $C$ is called the \emph{output gate}, and we write $C(x) = G_0(x)$. Thus $C$ computes a function from $2^\NN$ to $\{0,1\}$, and we say that $C$ \emph{computes} the set $C^{-1}(\{1\})$. We call two circuits \emph{equivalent} if they compute the same function.

\end{definition}

\begin{remark}

Allowing for $\NOT$-gates would not increase the power of these circuits, since De Morgan's law allows us to propagate all negations to the bottom layer of the circuit. Similarly, allowing circuits to be directed graphs instead of trees would not grant any additional power, since one can duplicate any subcircuits whose root nodes are children of multiple gates.

\end{remark}

\subsection{Essential Rank}

The typical notion of ordinal rank for well-founded trees gives us a measure of complexity for circuits, but we define a slightly different measure so that Lemma~\ref{clopen} and Corollary~\ref{rank-vs-circuit} hold.

\begin{definition}
\label{def:essential-rank}

If $T$ is a well-founded tree, we define the \emph{essential rank} $\rank(v)$ of each node $v$ in $T$ inductively as follows: If the subtree of $T$ rooted at $v$ is finite, we set $\rank(v) = 0$. Otherwise, we define
\begin{equation*}
    \rank(v) = \sup \{\rank(w) + 1 : w \text{ is a child of } v\}
\end{equation*}
We define $\rank(T)$ to be the essential rank of the root node of $T$.

\end{definition}

\begin{definition}

The \emph{rank} of a circuit $C$ is the rank of its underlying tree.

\end{definition}

\subsection{Equivalence of Borel sets and circuits}
\label{sec:correspondence}

The following result and its corollaries allow us to describe Borel subsets of $2^\NN$ using circuits.

\begin{lemma}
\label{clopen}

A set $S \subseteq 2^\NN$ is clopen if and only if it is computed by a finite circuit.

\end{lemma}

\begin{proof}

Any open set $S \subseteq 2^\NN$ can be written as a union of basic open sets, and if $S$ is also closed, this union can be taken to be finite by compactness. Each basic open set in such a union is computed by a finite circuit, so the forward direction follows by taking a finite disjunction of such circuits.

Conversely, any circuit consisting of a single input gate computes either $\emptyset$, $2^\NN$, or a set of the form $\{x \in 2^\NN : x_i = b\}$ for some $i \in \NN$ and $b \in \{0,1\}$. Note that these sets are all clopen. As the clopen sets are closed under finite unions and intersections, this means that any set $S$ computed by a finite circuit is also clopen.
\end{proof}

This lemma tells us that a subset of $2^\NN$ is clopen if and only if it is computed by a rank-$0$ circuit. We obtain the following corollary from induction:

\begin{corollary}
\label{rank-vs-circuit}

For any $\alpha < \omega_1$, a set $S \subseteq 2^\NN$ lies in $\BSigma^0_\alpha \cup \BPi^0_\alpha$ if and only if it is computed by a rank-$\alpha$ circuit.

\end{corollary}

\begin{corollary}
\label{circuit-borel}

A subset of $2^\NN$ is Borel if and only if it is computed by a circuit.

\end{corollary}

\subsection{Nondeterminism}
\label{sec:nondeterminism}

Just as in the finite setting, we can consider \emph{nondeterministic circuits} which take as input two sets of variables $x_0, x_1, \dots$ and $y_0, y_1, \dots$. We call $y_0, y_1, \dots$ the \emph{nondeterministic inputs}, and we say that a nondeterministic circuit $C$ \emph{computes} the set
\begin{equation*}
    \{x \in 2^\NN : C(x,y) = 1 \text{ for some } y \in 2^\NN\} \text{.}
\end{equation*}
It is immediate from Corollary~\ref{circuit-borel} that the analytic sets are precisely the sets computed by nondeterministic circuits. In fact, we can make the stronger claim that every analytic set is computed by a depth-$2$ nondeterministic circuit. The argument works just as in the proof that circuit satisfiability is polynomial-time reducible to $3\mathrm{SAT}$, nondeterministically guessing the values of the internal gates in the circuit and verifying that the resulting computation is valid. Since this result will be of crucial importance later, we include a proof for completeness:

\begin{theorem}
\label{depth-2}

A set $A \subseteq 2^\NN$ is analytic if and only if it is computed by a nondeterministic CNF (i.e.\ a depth-$2$ nondeterministic circuit whose output gate is an $\AND$-gate).

\end{theorem}

\begin{proof}

The reverse direction is immediate. For the forward direction, let $C$ be a nondeterministic circuit computing $A$ whose output gate is an internal gate (otherwise the result is trivial). Let $(G_n)_{n \in \NN}$ be an enumeration of the gates of $C$, with $G_0$ being the output gate. We define new nondeterministic variables $z_i$ for each $i \in \NN$ such that $G_i$ is an internal gate. For every $n \in \NN$, let $w_n$ denote the variable $z_n$ if $G_n$ is an internal gate and the label of $G_n$ otherwise. We then have $x \in A$ if and only if there exist $y,z \in 2^\NN$ such that $z_0 = 1$ and the following hold for each $i \in \NN$ with $G_i$ being an internal gate:
\begin{itemize}
    \item If $G_i$ is an $\OR$-gate with children $G_{n_0},G_{n_1},\dots$, then we have $z_i \OR \overline{w}_{n_j} = 1$ for each $j \in \NN$ and also $\overline{z}_i \OR \bigvee_{k \in \NN} w_{n_k} = 1$.
    \item If $G_i$ is an $\AND$-gate with children $G_{n_0},G_{n_1},\dots$, then we have $z_i \OR \bigvee_{j \in \NN} \overline{w}_{n_j} = 1$ and also $\overline{z}_i \OR w_{n_k} = 1$ for each $k \in \NN$.
\end{itemize}
By taking the conjunction of $z_0$ and all the clauses described above for each internal gate $G_i$, we obtain our desired nondeterministic CNF computing $A$ (up to relabeling).
\end{proof}

\begin{remark}

Because of the existential quantifier and the monotonicity of $\AND$- and $\OR$-gates, we still get an equivalent CNF if the clauses $z_i \OR \overline{w}_{n_j}$ for $j \in \NN$ and $z_i \OR \bigvee_{j \in \NN} \overline{w}_{n_j}$ are omitted.

\end{remark}

\begin{remark}

This result can be rephrased as saying that every analytic subset of $2^\NN$ is the projection of a $G_\delta$ (i.e.\ $\BPi^0_2$)-subset of $2^\NN \times 2^\NN$. Our combinatorial proof is quite different from the standard topological argument (and perhaps more elementary).

\end{remark}

\section{Infinite communication complexity}
\label{sec:infinite-cc}

In this section, we aim to define the communication complexity of functions from $2^\NN \times 2^\NN$ to $\{0,1\}$ or $\NN$. One immediate observation is that our protocol trees must be infinite in order to study any interesting problems. In order to ensure that every protocol computes a well-defined function, we also require that protocols be well-founded. By K\H{o}nig's lemma, any well-founded tree which is finitely branching must itself be finite. For this reason, we allow protocol trees to have countably infinite branching. In other words, Alice and Bob communicate by sending natural numbers instead of bits.

\begin{definition}[Infinite communication protocols]
\label{def:infinite-cc}

Let $X = Y = 2^\NN$ and $Z$ be countable. A \emph{protocol} $P$ for a communication problem $f \colon X \times Y \rightarrow Z$ is a countable, well-founded tree with internal nodes $v$ labeled by functions $a_v \colon X \rightarrow \NN$ or $b_v \colon Y \rightarrow \NN$ and leaf nodes labeled by elements of $Z$. Every edge is labeled with some $n \in \NN$. For each internal node $v$ and natural number $n$ in the range of $a_v$ or $b_v$, exactly one edge connecting $v$ to one of its children must be labeled with $n$. The function computed by $P$ is defined analogously to the finite case (see Section~\ref{sec:finite-cc}).

\end{definition}

In the finite setting, every problem has a communication protocol (Alice or Bob can just send the other person their string in its entirety). In the infinite setting, on the other hand, this is not the case (we will prove this shortly). While the cost of a communication protocol is an essential notion in the finite case, the most interesting question in the infinite setting is often whether a protocol even exists. Nonetheless, we can still define the cost of protocols and communication complexity of functions in the infinite setting. Essential rank is used in place of the ordinary notion of rank so that Corollary~\ref{cost-vs-rank} holds.

\begin{definition}[Infinite communication complexity]

The \emph{cost} of a protocol is the essential rank (see Definition~\ref{def:essential-rank}) of the underlying tree. If a function $f \colon X \times Y \rightarrow Z$ has a communication protocol, the \emph{communication complexity} $D(f)$ of $f$ is the minimum cost of a protocol computing $f$.

\end{definition}

\subsection{Lower bounds}

Just as in the finite case, we can use partition numbers to prove lower bounds on infinite communication problems:

\begin{theorem}
\label{partition-bound}

If $f\colon X \times Y \rightarrow Z$ has a communication protocol, then $X \times Y$ can be partitioned into countably many $f$-monochromatic rectangles.

\end{theorem}

\begin{proof}

Suppose $f$ has a protocol $P$. The same argument as in the finite case tells us that the set of inputs reaching a given leaf node $v$ form an $f$-monochromatic rectangle $R_v$. Since $P$ has countably many leaf nodes, the collection of sets $R_v$ for all leaf nodes $v$ is a countable partition of $X \times Y$ into $f$-monochromatic rectangles.
\end{proof}

\begin{corollary}
\label{eq-lower-bound}

The function $\EQ \colon 2^\NN \times 2^\NN \rightarrow \{0,1\}$ defined by $\EQ(x,y) = 1$ if $x = y$ and $\EQ(x,y) = 0$ otherwise has no communication protocol.

\end{corollary}

\begin{proof}

The set $\mathcal{F} = \{(x,x) : x \in 2^\NN\}$ is uncountable, and by the same argument as in the finite case, any $\EQ$-monochromatic rectangle contains at most one element of $\mathcal{F}$. This means that any $\EQ$-monochromatic partition of $X \times Y$ requires uncountably many rectangles, and the contrapositive of Theorem~\ref{partition-bound} gives us the desired result.
\end{proof}

\begin{remark}
\label{rem:disj-lower-bound}

If we define $\DISJ \colon 2^\NN \times 2^\NN \rightarrow \{0,1\}$ by $\DISJ(x,y) = 1$ if and only if $x_i \wedge y_i = 0$ for all $i \in \NN$, then the fooling set $\mathcal{F} = \{(x,\overline{x}) : x \in 2^\NN\}$ tells us that $\DISJ$ also has no communication protocol.

\end{remark}

\begin{remark}
\label{efficient-computation}

Infinite communication problems with communication protocols seem to correspond loosely with finite communication problems with \emph{efficient} protocols (in the sense that the communication cost grows polylogarithmically with the input length). The functions $\EQ$ and $\DISJ$ give us examples of problems which are hard in both settings. For an example of an easy communication problem with an easy finite analog, consider the \emph{finite disjointness} function $\DISJ_{<\NN}$, which outputs $1$ whenever $x$ and $y$ satisfy $\DISJ(x,y) = 1$ and each contain finitely many ones. The finite analog is the \emph{$k$-disjointness} function $\DISJ_{k,n}$ (for fixed $k$), which outputs $1$ whenever $x$ and $y$ satisfy $\DISJ_n(x,y) = 1$ and each contain at most $k$ ones \cite{nisan97}.

\end{remark}

The correspondence in the previous remark gives us heuristic evidence that Definition~\ref{def:infinite-cc} is a reasonable extension of communication complexity to the infinite setting. The next section provides more concrete evidence that our definition of infinite communication complexity is useful in descriptive set theory.

\section{Infinite Karchmer--Wigderson games}
\label{sec:kw}

In \cite{kw1990}, Karchmer and Wigderson define for each function $f \colon \{0,1\}^n \rightarrow \{0,1\}$ a communication problem which completely characterizes the minimum depth of a (finite, fan-in $2$) circuit computing $f$. This communication problem is called the \emph{Karchmer--Wigderson (KW) game} and is a useful tool in proving circuit lower bounds. With our definition of infinite communication complexity and characterization of Borel sets via infinite circuits, we can prove connections between Borel sets and a certain infinite analog of KW games.

\begin{definition}[Karchmer--Wigderson games]

Fix two sets $A,B \subseteq 2^\NN$. In the \emph{KW game} $\KW_{A,B}$, Alice is given a string $x \in A$, Bob is given a string $y \in B$, and they must output an index $i \in \NN$ such that $x_i \neq y_i$.

\end{definition}

\begin{remark}

We do not require in the above definition that $A$ and $B$ are disjoint (though this is obviously a necessary condition for $\KW_{A,B}$ to be solvable).

\end{remark}

\begin{remark}

Unlike many games in descriptive set theory, KW games (along with all other communication games) are \emph{cooperative,} rather than competitive.

\end{remark}

\begin{remark}

The game $\KW_{A,B}$ defines a relation rather than a function, since there can be multiple indices $i \in \NN$ for which $x_i \neq y_i$. We say that a protocol $P$ solves the game $\KW_{A,B}$ if the $P(x,y)$-th bit of $x$ and $y$ disagree for all $(x,y) \in A \times B$.

\end{remark}

\begin{definition}

We define the \emph{communication complexity} $D(\KW_{A,B})$ of a KW game $\KW_{A,B}$ as the minimum cost of a protocol solving $\KW_{A,B}$.

\end{definition}

\begin{remark}
\label{rem:always-terminating}

When thinking about KW games, it is crucial to distinguish between always-terminating procedures and well-founded protocols. The game $\KW_{A,B}$ has an always-terminating procedure whenever $A \cap B = \emptyset$, since Alice and Bob can first confirm that $(x,y) \in A \times B$ and then share the bits of $x$ and $y$ one-by-one until finding a bit where the two strings disagree. However, this procedure is generally ill-founded and thus not a protocol.

\end{remark}

\begin{remark}
\label{rem:closed-kw-games}

The procedure in Remark~\ref{rem:always-terminating} is a well-founded protocol when $A$ and $B$ are closed and disjoint. This is because any infinite branch would give us a point in the closure of both $A$ and $B$, contradicting the fact that $A \cap B = \emptyset$. K\H{o}nig's lemma tells us that finitely-branching protocols are themselves finite, so we can conclude that $D(\KW_{A,B}) = 0$ whenever $A$ and $B$ are closed and disjoint.

\end{remark}

\subsection{KW games and Borel separability}

We now prove Theorem~\ref{borel-kw-equivalence}, the central connection between KW games and Borel sets. This link is the motivation for our definitions of infinite communication complexity and KW games, and it is the essential piece of machinery for our upcoming results.

\begin{definition}

Let $A,B \subseteq 2^\NN$. We say that a set $S \subseteq 2^\NN$ \emph{separates} $A$ and $B$ if $A \subseteq S$ and $S \cap B = \emptyset$. We call $A$ and $B$ \emph{Borel separable} whenever $A$ and $B$ can be separated by a Borel set.

\end{definition}

Recall that Theorem~\ref{borel-kw-equivalence} states that $A$ and $B$ are Borel separable if and only if $\KW_{A,B}$ admits a communication protocol. We are now ready to prove this fact.

\begin{proof}[Proof of Theorem~\ref{borel-kw-equivalence}]

First suppose that $A$ and $B$ are Borel separable. Then by Corollary~\ref{circuit-borel}, there exists a circuit $C$ such that $C(x) = 1$ for all $x \in A$ and $C(y) = 0$ for all $y \in B$. Consider the following protocol for $\KW_{A,B}$: Given $x \in A$ and $y \in B$, Alice and Bob begin by examining the output gate $G$ of $C$, which satisfies $G(x) = 1$ and $G(y) = 0$. Let $(G_n)_{n \in \NN}$ denote the children of $G$. If $G$ is an $\OR$-gate, then we have $G_j(x) = 1$ for some $j \in \NN$ and $G_n(y) = 0$ for all $n \in \NN$. Alice shares $j$ with Bob in this case. If $G$ is an $\AND$-gate, then we have $G_n(x) = 1$ for all $n \in \NN$ and $G_k(y) = 0$ for some $k \in \NN$. Bob shares $k$ with Alice in this case. In both cases, Alice and Bob find a descendant $G'$ of $G$ such that $G'(x) = 1$ and $G'(y) = 0$. Alice and Bob repeat this process until finding a literal $x_i$ or $\overline{x}_i$ evaluating to $1$ on $x$ and $0$ on $y$. Thus Alice and Bob can guarantee correctness by outputting $i$. Since $C$ is well-founded, this procedure is also well-founded and gives us a protocol for $\KW_{A,B}$.

Now suppose $\KW_{A,B}$ has a protocol. For each node $v$ in $P$, let $R_v = A_v \times B_v \subseteq A \times B$ denote the set of inputs $(x,y) \in A \times B$ on which the protocol reaches $v$ (the argument in Theorem~\ref{partition-bound} tells us that these sets are indeed rectangles). Suppose first that $v$ is a leaf node. Letting $i$ denote the label of $v$, the correctness of $P$ tells us that $x_i \neq y_i$ for all $(x,y) \in R_v$. As $R_v$ is a rectangle, this means that either $x_i = 1$ and $y_i = 0$ for all $(x,y) \in R_v$ or $x_i = 0$ and $y_i = 1$ for all $(x,y) \in R_v$. Call $v$ a \emph{positive leaf} if the first case holds and a \emph{negative leaf} otherwise.

Now replace each of Alice's nodes with an $\OR$-gate, each of Bob's nodes with an $\AND$-gate, each positive leaf labeled $i$ with the literal $x_i$, and each negative leaf labeled $i$ with the literal $\overline{x}_i$. The resulting circuit $C$ is well-founded since $P$ is a well-founded protocol. It therefore suffices to show that $C(x) = 1$ for all $x \in A$ and $C(y) = 0$ for all $y \in B$. To this end, let $G_v$ denote the gate corresponding to the node $v$ in $P$. We claim that, for each vertex $v$ in $P$, we have $G_v(x) = 1$ for all $x \in A_v$ and $G_v(y) = 0$ for all $y \in B_v$. We saw already that the claim holds when $v$ is a leaf node. If $v$ is an internal node, we may inductively suppose that the claim holds for all children $v_0,v_1,\dots$ of $v$. If $v$ is one of Alice's nodes, then for any $x \in A_v$, we have $x \in A_{v_j}$ for some $j \in \NN$. Thus $G_{v_j}(x) = 1$ by the inductive hypothesis, implying that $G_v(x) = \bigvee_{j \in \NN} G_{v_j}(x) = 1$. Similarly, if $v$ is one of Bob's nodes, then every $y \in B_v$ satisfies $G_{v_k}(y) = 0$ for some $k \in \NN$ and hence $G_v(y) = \bigwedge_{k \in \NN} G_{v_k}(y) = 0$. Thus the claim holds for $v$, completing the induction. The root node $v$ of $P$ satisfies $R_v = A \times B$, so applying the claim to this node completes the proof.
\end{proof}

\begin{corollary}
\label{borel-kw-corollary}

A set $S$ is Borel if and only if $\KW_{S,S^c}$ has a communication protocol.

\end{corollary}

The conversion process in the previous proof exactly preserves the tree structure underlying the protocols and circuits involved. As the essential rank of this underlying tree is equal to both the cost of the protocol and the rank of the circuit, we obtain the following corollary:

\begin{corollary}
\label{cost-vs-rank}

If $A$ and $B$ are Borel separable subsets of $2^\NN$, then the minimum rank of a Borel set separating $A$ and $B$ is precisely $D(\KW_{A,B})$. The Borel rank of a Borel set $B$ is precisely $D(\KW_{B,B^c})$.

\end{corollary}

\begin{remark}

The definition of infinite communication complexity allows for the functions $a_v$ and $b_v$ to be arbitrary. Thus a protocol $P$ solving some KW game $\KW_{A,B}$ might not have its nodes labeled by Borel functions. Given such a protocol, however, we can use the previous proof to obtain a circuit $C$ separating $A$ and $B$ and then convert $C$ back to a protocol $P'$ (with $j$ and $k$ chosen minimally) solving $\KW_{A,B}$ with Borel functions as labels. The underlying tree structure is preserved in the process. This means that the communication complexity of $\KW_{A,B}$ is always ``witnessed'' by a protocol with internal nodes labeled by Borel functions.

\end{remark}

\begin{remark}

The observation that $D(\KW_{A,B}) = 0$ when $A$ and $B$ are closed (see Remark~\ref{rem:closed-kw-games}) gives us an alternative proof that disjoint closed subsets of $2^\NN$ are clopen-separable.

\end{remark}

\section{Existence results using KW games}
\label{sec:existence}

The rest of this paper is devoted to applications of infinite Karchmer--Wigderson games. In this section, we will prove positive results, demonstrating that Borel sets exist by solving the appropriate KW games. In particular, we will prove the analytic separation theorem and the fact that any monotone Borel function is computed by a monotone circuit.

\subsection{The analytic separation theorem}

The analytic separation theorem states that any two disjoint analytic sets are separable by a Borel set. In particular, if a set $A$ and its complement are both analytic, then $A$ must actually be a Borel set. KW games give us a completely elementary and purely combinatorial proof:

\begin{theorem}
\label{analytic-separation}

Any two disjoint analytic sets $A,A' \subseteq 2^\NN$ are Borel separable.

\end{theorem}

\begin{proof}

By Theorem~\ref{borel-kw-equivalence}, it suffices to find a communication protocol for $\KW_{A,A'}$. To this end, we use Theorem~\ref{depth-2} to find nondeterministic CNFs $C(x,y)$ and $C'(x',y')$ computing $A$ and $A'$, respectively. These circuits are shared knowledge between Alice and Bob, since they do not depend on $x$ or $x'$. Let $(G_n)_{n \in \NN}$ denote the clauses (i.e.\ rank-$1$ subcircuits) of $C$ and $(G_n')_{n \in \NN}$ denote the clauses of $C'$. We call a circuit or gate \emph{satisfied} if it evaluates to $1$. Note that a CNF is satisfied if and only if each of its clauses contains a satisfied literal.

Given $(x,x') \in A \times A'$, Alice and Bob first individually choose strings $y,y' \in 2^\NN$ such that $C(x,y) = C'(x',y') = 1$. Alice then tells Bob a literal in $G_0$ which is satisfied, Bob tells Alice a literal in $G_0'$ which is satisfied, Alice shares a literal in $G_1$ which is satisfied, etc. If Alice and Bob ever share both the literals $x_i$ and $\overline{x}_i'$ or both the literals $\overline{x}_i$ and $x_i'$, they output $i$.

This protocol always outputs a correct value, so we just need to show that it is well-founded. To this end, suppose for the sake of contradiction that the protocol has an infinite branch. Such an infinite branch specifies exactly one literal in each clause of $C$ or $C'$. Let $\mathcal{L}$ denote the collection of all literals selected along this branch. Because the protocol terminates whenever both the literals $x_i$ and $\overline{x}_i'$ or both the literals $\overline{x}_i$ and $x_i'$ are shared, we must have for each $i \in \NN$ that the set $\mathcal{L} \cap \{x_i,\overline{x}_i,x_i',\overline{x}_i'\}$ is a subset of either $\{x_i, x_i'\}$ or $\{\overline{x}_i, \overline{x}_i'\}$. We also have that for each $i \in \NN$, at most one of the literals $y_i$ and $\overline{y}_i$ is contained in $\mathcal{L}$, and at most one of the literals $y_i'$ and $\overline{y}_i'$ is contained in $\mathcal{L}$. We can therefore choose $\tilde{x},\tilde{y},\tilde{y}' \in 2^\NN$ such that every literal in $\mathcal{L}$ is satisfied when evaluating $C(\tilde{x},\tilde{y})$ and $C'(\tilde{x},\tilde{y}')$. Because $\mathcal{L}$ contains a literal in every clause of $C$ or $C'$, we have $C(\tilde{x},\tilde{y}) = C'(\tilde{x},\tilde{y}') = 1$, implying that $\tilde{x} \in A \cap A'$. This contradicts the assumption that $A$ and $A'$ are disjoint.
\end{proof}

\subsection{Monotone circuits for monotone Borel sets}

\begin{definition}

We write $x \leq y$ for binary strings $x$ and $y$ whenever $x_i \leq y_i$ for every index $i$. We call a function $f \colon 2^\NN \rightarrow \{0,1\}$ or $f \colon \{0,1\}^n \rightarrow \{0,1\}$ \emph{monotone} if $f(x) \leq f(y)$ whenever $x \leq y$. 

\end{definition}

\begin{definition}

A circuit is \emph{monotone} if it contains no negated literals $\overline{x}_i$ for $i \in \NN$.

\end{definition}

It is easy to show that a function $f \colon \{0,1\}^n \rightarrow \{0,1\}$ is monotone if and only if it is computed by some monotone circuit. However, the story changes when we only consider polynomial-sized circuits. In \cite{tardos1988}, Tardos proved that there are monotone functions computable by polynomial-sized circuit families but not by polynomial-sized families of monotone circuits. This was both a surprising result and a disappointment to complexity theorists hoping to use superpolynomial lower bounds on monotone circuits to resolve the $\Pclass$ vs.\ $\NP$ question.

Given the analogy between computability in the infinite setting and efficient computability in the finite setting \cite{parity81,sipser1983,sipser1984}, one might expect a similar phenomenon in the infinite setting. However, it turns out that all monotone Borel functions have monotone circuits (this follows from the Lyndon interpolation theorem \cite{lyndon1959, dyck90}, and Kechris gives a self-contained proof \cite{kechris1995}). In other words, if a circuit computes a monotone function, one can always find a monotone circuit computing the same function. We provide a new proof of this fact using a monotone version of KW games which characterize the functions computed by monotone circuits.

\begin{definition}

In the \emph{monotone KW game} $\mKW_{A,B}$, Alice is given a string $x \in A$, Bob is given a string $y \in B$, and they must output an index $i \in \NN$ such that $x_i = 1$ and $y_i = 0$.

\end{definition}

\begin{definition}

Given $A,B \subseteq 2^\NN$, we say that a function $f \colon 2^\NN \rightarrow \{0,1\}$ \emph{separates} $A$ and $B$ if $f(x) = 1$ for all $x \in A$ and $f(y) = 0$ for all $y \in B$.

\end{definition}

\begin{theorem}
\label{monotone-kw-equivalence}

Two sets $A,B \subseteq 2^\NN$ can be separated by a monotone circuit if and only if $\mKW_{A,B}$ has a communication protocol.

\end{theorem}

\begin{proof}

If $C$ is a monotone circuit separating $A$ and $B$, then the protocol in the proof of Theorem~\ref{borel-kw-equivalence} outputs $i \in \NN$ such that some input gate in $C$ labeled with $x_i$ or $\overline{x}_i$ evaluates to $1$ on $x$ and $0$ on $y$. As $C$ is monotone, this literal must be $x_i$, implying that $x_i = 1$ and $y_i = 0$. Conversely, if $P$ is a protocol for $\mKW_{A,B}$, then every leaf node $v$ in $P$ is a positive leaf. Thus the conversion process in Theorem~\ref{borel-kw-equivalence} gives us a monotone circuit separating $A$ and $B$.
\end{proof}

We can now use this characterization to provide a simple proof that monotone Borel functions have monotone circuits.

\begin{theorem}
\label{monotone-borel}

Every monotone Borel function $f \colon 2^\NN \rightarrow \{0,1\}$ is computed by a monotone circuit.

\end{theorem}

\begin{proof}

We prove the stronger claim that any two analytic sets $A,A' \subseteq 2^\NN$ separable by a monotone function are separable by a monotone circuit. By Theorem~\ref{monotone-kw-equivalence}, it suffices to find a communication protocol for $\mKW_{A,A'}$. We use the same game as in the proof of Theorem~\ref{analytic-separation}, except that Alice and Bob keep playing even if they share both the literals $\overline{x}_i$ and $x_i'$ for some $i \in \NN$ (i.e.\ the only stopping condition is sharing both $x_i$ and $\overline{x}_i'$).

We again see that this protocol always outputs a correct value, and to prove well-foundedness, suppose for the sake of contradiction that the protocol has an infinite branch. A similar argument to that in the proof of Theorem~\ref{analytic-separation} lets us choose $\tilde{x},\tilde{x}',\tilde{y},\tilde{y}' \in 2^\NN$ such that $C(\tilde{x},\tilde{y}) = C'(\tilde{x}',\tilde{y}') = 1$ and $\tilde{x} \leq \tilde{x}'$. As $\tilde{x} \in A$ and $\tilde{x}' \in A'$, this contradicts the monotone-separability of $A$ and $A'$.
\end{proof}

\section{Impossibility results using KW games}
\label{sec:impossibility}

While the previous section focuses on positive results about Borel sets, this section focuses on negative results. We first discuss why many techniques from finite communication complexity are insufficient to prove that a given KW game has no communication protocol. Nonetheless, we will prove that the KW games for infinite parity and ill-founded trees lack protocols and thereby present new, combinatorial proofs that these sets are not Borel.

\subsection{Monochromatic partitions for KW games}
\label{sec:partitions}

Most lower bounds in finite communication complexity involve proving lower bounds on the number of $f$-monochromatic rectangles needed to partition a domain. We saw in Corollary~\ref{eq-lower-bound} and Remark~\ref{rem:disj-lower-bound} that partition number lower bounds are useful in the infinite setting as well.

Aho, Ullman, and Yannakakis' results in \cite{auy} show that, for a boolean-valued function $f$ with partition number $\chi(f)$, we have $D(f) \leq O\!\left((\log_2\chi(f))^2\right)$. Yannakakis extended these ideas in \cite{yannakakis1991} to show that the same bound holds for all functions $f$, even those which are not boolean-valued. This means that, up to a quadratic loss, proving lower bounds on communication complexity in the finite setting amounts to proving lower bounds on partition numbers. Indeed, a wide array of lower bound techniques in communication complexity follow this approach.

One might expect that a similar result holds in the infinite setting. We proved that any communication protocol for $f$ gives us a countable partition of $2^\NN \times 2^\NN$ into $f$-monochromatic rectangles. Unfortunately, the converse fails dramatically. In particular, the existence of non-Borel sets tells us that there are KW games without communication protocols. Despite this fact, every KW game can be solved by a function $f$ admitting a countable $f$-monochromatic partition of $2^\NN \times 2^\NN$:

\begin{theorem}
\label{kw-partition}

If $A,B \subseteq 2^\NN$ are disjoint, then there is a function $f \colon 2^\NN \times 2^\NN \rightarrow \NN$ solving $\KW_{A,B}$ such that $2^\NN \times 2^\NN$ has a countable partition into $f$-monochromatic rectangles.

\end{theorem}

\begin{proof}

We define $f(x,y) = 0$ when $x \notin A$ or $y \notin B$, and we set $f(x,y)$ to be the first index where $x$ and $y$ differ if $x \in A$ and $y \in B$. Then $2^\NN \times 2^\NN$ is partitioned by the rectangles $A^c \times B$, $A \times B^c$, $A^c \times B^c$, and
\begin{equation*}
    (A \cap wb2^\NN) \times (B \cap w\overline{b}2^\NN)
\end{equation*}
for $w \in 2^{<\NN}$ and $b \in \{0,1\}$.
\end{proof}

\begin{remark}

This countable partition corresponds exactly to the always-terminating procedure for every KW game discussed in Remark~\ref{rem:always-terminating}.

\end{remark}

\subsection{The Borel cut-and-choose game}
\label{sec:bcg}

The previous section tells us that, in order to obtain impossibility results using KW games, our techniques must go beyond counting monochromatic rectangles. In other words, we must take advantage of the fact that protocols are well-founded, not just always-terminating. We will do so using the following infinite adversarial game.

\begin{definition}

The \emph{Borel cut-and-choose game} $\BCG_{A,B}$ takes as parameters two subsets $A$ and $B$ of $2^\NN$. It is played between Player I (Cut) and Player II (Choose) as follows:
\begin{itemize}
    \item We begin with $A^{(0)} = A$ and $B^{(0)} = B$.
    \item In the $n$-th round, Player I has the option to either partition $A^{(n)}$ or $B^{(n)}$ into countably many pieces $A^{(n)}_0, A^{(n)}_1, \dots$ or $B^{(n)}_0, B^{(n)}_1, \dots$.
    \begin{itemize}
        \item If Player I partitioned $A$, Player II chooses some $j \in \NN$ and we take $A^{(n+1)} = A^{(n)}_j$ and $B^{(n+1)} = B^{(n)}$.
        \item If Player I partitioned $B$, Player II chooses some $k \in \NN$ and we take $A^{(n+1)} = A^{(n)}$ and $B^{(n+1)} = B^{(n)}_k$.
    \end{itemize}
    \item If there exist $i,n \in \NN$ such that $x_i \neq y_i$ for all $(x,y) \in A^{(n)} \times B^{(n)}$, then Player I wins. Otherwise, Player II wins.
\end{itemize}

\end{definition}

The name is chosen based on the resemblance of this game to those popularized by Jech in \cite{jech1984}. Borel cut-and-choose games are open and therefore determined. These games characterize Borel separability via the following theorem:

\begin{theorem}
\label{bcg}

Two sets $A,B \subseteq 2^\NN$ are Borel separable if and only if Player I has a winning strategy in the game $\BCG_{A,B}$.

\end{theorem}

\begin{proof}

For the reverse direction, we prove the contrapositive. If $A$ and $B$ are not Borel separable, then Player II can always play so that $A^{(n)}$ and $B^{(n)}$ are not Borel separable for each $n \in \NN$ (this is because if $\bigcup_{j \in \NN} S_j$ and $T$ are not Borel separable, then $S_j$ and $T$ are not Borel separable for some $j \in \NN$). In particular, Player II can ensure that $A^{(n)}$ and $B^{(n)}$ are never separable by a set of the form $\{x \in 2^\NN : x_i = b\}$ for $i \in \NN$ and $b \in \{0,1\}$. Thus Player II has a winning strategy.

For the forward direction, we use KW games. If $A$ and $B$ are Borel separable, then Theorem~\ref{borel-kw-equivalence} gives us a communication protocol $P$ for $\KW_{A,B}$. Recall that for a node $v$ in $P$, we define $R_v = A_v \times B_v$ to be the set of pairs $(x,y) \in A \times B$ such that the protocol reaches $v$ on input $(x,y)$. Player I's strategy is to always ensure that $A^{(n)}$ and $B^{(n)}$ correspond to the sets $A_{v^{(n)}}$ and $B_{v^{(n)}}$ for some node $v^{(n)} \in P$ which is a distance of $n$ from the root node. In particular, if $v^{(n)}$ is one of Alice's nodes with children $v_0, v_1,\dots$, Player I partitions $A^{(n)} = A_{v^{(n)}}$ into the sets $A_{v_0}, A_{v_1},\dots$. If, on the other hand, $v^{(n)}$ is one of Bob's nodes with children $v_0, v_1,\dots$, Player I partitions $B^{(n)} = B_{v^{(n)}}$ into the sets $B_{v_0}, B_{v_1},\dots$. We eventually reach a leaf node by the well-foundedness of $P$, at which point Player I wins.
\end{proof}

We can also formulate a monotone version of Borel cut-and-choose games:

\begin{definition}

The \emph{monotone Borel cut-and-choose game} $\mBCG_{A,B}$ is defined identically to the non-monotone version, except that the winning condition is changed: Player I wins if there exist $i,n \in \NN$ such that $x_i = 1$ for all $x \in A^{(n)}$ and $y_i = 0$ for all $y \in B^{(n)}$, and Player II wins otherwise.

\end{definition}

\begin{theorem}
\label{monotone-bcg}

Two sets $A,B \subseteq 2^\NN$ are separable by a monotone Borel set if and only if Player I has a winning strategy in the game $\mBCG_{A,B}$.

\end{theorem}

\begin{proof}

The reverse direction is proved identically to that of Theorem~\ref{bcg}, using the fact that if $\bigcup_{j \in \NN} S_j$ and $T$ are not separable by a monotone Borel set, then $S_j$ and $T$ are not separable by a monotone Borel set for some $j \in \NN$.

For the forward direction, suppose $A$ and $B$ are separable by a monotone Borel set. Then by Theorem~\ref{monotone-borel}, there exists a monotone circuit separating $A$ and $B$. Thus Theorem~\ref{monotone-kw-equivalence} gives us a protocol $P$ for $\mKW_{A,B}$, and the rest of the proof proceeds identically to that of Theorem~\ref{bcg}.
\end{proof}

\begin{remark}

The forward directions of Theorems~\ref{bcg} and \ref{monotone-bcg} give us a technique to prove that two sets $A$ and $B$ are not (monotone) Borel separable: providing a winning strategy for Player II in $\BCG_{A,B}$ or $\mBCG_{A,B}$. We can interpret this strategy as traversing down an arbitrary protocol tree in such a manner that each node $v$ we reach contains for every $i \in \NN$ some pair $(x,y) \in R_v$ with $x_i = y_i$. Since we must eventually reach a leaf node by well-foundedness, the strategy ensures that our arbitrarily-chosen protocol does not solve $\KW_{A,B}$.

\end{remark}

To obtain a winning strategy for Player II, we will define a notion of ``largeness'' so that the strategy can simply be to pick a ``large'' set. For this to work, it must be the case that if $\bigcup_{n \in \NN} S_n$ is large, so is some $S_n$. Baire category will give us one such notion of largeness.

\subsection{Inseparability of everywhere nonmeager sets}
\label{sec:baire}

In this section, we let $\Gamma$ denote either the binary alphabet $\{0,1\}$ or the natural numbers $\NN$. Given a string $v \in \Gamma^{<\NN}$, we use the notation $[v]$ for the basic open set $v\Gamma^\NN$. The following definitions are phrased combinatorially, but they are equivalent to the standard topological formulations.

\begin{definition}

Call a set $A \subseteq \Gamma^\NN$ \emph{nowhere dense} if for every $v \in \Gamma^{<\NN}$, there exists $w \in \Gamma^{<\NN}$ such that $[vw] \cap A = \emptyset$. A set is \emph{meager} if it can be written as countable union of nowhere dense sets. A set is \emph{nonmeager} if it is not meager and \emph{nonmeager in $[v]$} if its intersection with $[v]$ is nonmeager. A set is \emph{everywhere nonmeager in $[v]$} if its intersection with $[vw]$ is nonmeager for every $w \in \Gamma^{<\NN}$.

\end{definition}

Note that if $A$ is everywhere nonmeager in $[v]$, then it is also everywhere nonmeager in $[vw]$ for all $w \in \Gamma^{<\NN}$. The Baire category theorem tells us that $\Gamma^\NN$ is everywhere nonmeager. The following lemma is the key fact about everywhere nonmeager sets which lets us call being everywhere nonmeager a notion of largeness.

\begin{lemma}
\label{everywhere-nonmeager}

If $A \subseteq \Gamma^\NN$ is nonmeager in $[v]$, then there exists $w \in \Gamma^{<\NN}$ such that $A$ is everywhere nonmeager in $[vw]$.

\end{lemma}

\begin{proof}

Suppose not. Then for every $w \in \Gamma^{<\NN}$, there exists $z_w \in \Gamma^{<\NN}$ such that $A \cap [vwz_w]$ is meager. Thus the set $A \cap \bigcup_{w \in \Gamma^{<\NN}}[vwz_w]$ is meager. We also claim that the set $B = [v] \setminus \left(\bigcup_{w \in \Gamma^{<\NN}}[vwz_w]\right)$ is nowhere dense. To this end, fix a string $s \in \Gamma^{<\NN}$. If we can write $s$ in the form $vt$ for some $t \in \Gamma^{<\NN}$, then we have $sz_t = vtz_t \in \bigcup_{w \in \Gamma^{<\NN}}[vwz_w]$ and hence $sz_t \notin B$. Otherwise, we can choose $z \in \Gamma^{<\NN}$ such that $sz \notin [v]$ and hence $sz \notin B$. This proves that $B$ is nowhere dense, and we have in particular that $A \cap B$ is meager. Thus $A \cap [v]$ is meager as the union of the meager sets $A \cap \bigcup_{w \in \Gamma^{<\NN}}[vwz_w]$ and $A \cap B$, contradicting our assumption.
\end{proof}

Now let us see how we can use everywhere nonmeager sets to provide a winning strategy in the Borel cut-and-choose game and prove our first impossibility result:

\begin{theorem}
\label{nonmeager-inseparability}

If $A \subseteq 2^\NN$ is everywhere nonmeager and $B \subseteq 2^\NN$ is nonmeager, then $A$ and $B$ are not Borel separable.

\end{theorem}

\begin{proof}

By Theorem~\ref{bcg}, it suffices to exhibit a winning strategy for Player II in the game $\BCG_{A,B}$. We claim that Player II can win by always maintaining the invariant that, for all $n \in \NN$, the sets $A^{(n)}$ and $B^{(n)}$ are everywhere nonmeager in $[w^{(n)}]$ for some $w^{(n)} \in 2^{<\NN}$.

Lemma~\ref{everywhere-nonmeager} lets us choose $w^{(0)} \in 2^{<\NN}$ such that $B$ is everywhere nonmeager in $[w^{(0)}]$, so the invariant initially holds. We now show inductively that Player II's invariant can always be maintained, so suppose $A^{(n)}$ and $B^{(n)}$ are everywhere nonmeager in $[w^{(n)}]$. If $A^{(n)} = \bigcup_{j \in \NN} A^{(n)}_j$, then some $A^{(n)}_j$ must be nonmeager in $[w^{(n)}]$. Thus Lemma~\ref{everywhere-nonmeager} gives us a string $z \in 2^{<\NN}$ with $A^{(n)}_j$ everywhere nonmeager in $[w^{(n)}z]$. Player II chooses this $j$, and the invariant holds for $w^{(n+1)} = w^{(n)}z$. A symmetric argument tells us that this is also possible when Player I chooses to partition $B^{(n)}$. This completes the induction.

Finally, fix $i,n \in \NN$. We know that $A^{(n)}$ and $B^{(n)}$ are everywhere nonmeager in $[w^{(n)}]$ for some $w^{(n)} \in 2^{<\NN}$, and we extend $w^{(n)}$ to some string $w$ of length at least $i$. As $A^{(n)}$ and $B^{(n)}$ are everywhere nonmeager in $[w]$, these sets both have nonempty intersection with $[w]$. Thus there exist strings $x \in A^{(n)}$ and $y \in B^{(n)}$ with $x_i = y_i$, so the winning condition for Player I does not hold. As $i,n \in \NN$ were arbitrary, this strategy is winning for Player II.
\end{proof}

Theorem~\ref{nonmeager-inseparability} can also be proven using the fact that Borel sets are Baire-measurable, so the Borel cut-and-choose game is no more powerful than standard topological arguments when one only uses everywhere nonmeager sets as an inductive invariant. There are, however, other notions of largeness one can use. Our proof in the next section that the set of ill-founded trees is not Borel uses such an alternative notion, and the result is beyond the reach of standard topological arguments in descriptive set theory.

\subsubsection{Hardness of parity}

A function $f \colon 2^\NN \rightarrow \{0,1\}$ is a \emph{parity function} if $f(x) \neq f(y)$ whenever $x$ and $y$ differ in a single bit. One consequence of Theorem~\ref{nonmeager-inseparability} is proving that no such function is Borel. The only fact we will need about parity functions is the following Baire-categorical property:

\begin{lemma}
\label{parity-nonmeagerness}

If $f \colon 2^\NN \rightarrow \{0,1\}$ is a parity function, then the sets $f^{-1}(\{1\})$ and $f^{-1}(\{0\})$ are everywhere nonmeager in $2^\NN$.

\end{lemma}

This lemma follows from combining the Baire category theorem with the observation that, for any $w \in 2^{<\NN}$, the set $f^{-1}(\{1\}) \cap [w]$ is meager if and only if the set $f^{-1}(\{0\}) \cap [w]$ is meager. The lemma lets us apply Theorem~\ref{nonmeager-inseparability} to obtain our desired result:

\begin{corollary}
\label{parity}

No infinite parity function is Borel.

\end{corollary}

\subsection{Hardness of ill-founded trees}

To prove that the set of ill-founded trees is not Borel, we need notions of largeness for ill-founded and well-founded trees. For ill-founded trees, we use a variant of being everywhere nonmeager. Let $\NN^\NN$ denote the Baire space, and recall that trees are represented as prefix-closed subsets of $\NN^{<\NN}$.

\begin{definition}

For a collection $A$ of ill-founded trees, write $X_A \subseteq \NN^\NN$ to denote the set of strings $x$ such that $A$ contains the tree consisting precisely of the finite prefixes of $x$ (in other words, such a tree contains a single infinite branch and no other vertices).

\end{definition}

This definition gives us a notion of largeness for ill-founded trees, since we can ask whether $X_A$ is everywhere nonmeager in $[w]$ for a string $w \in \NN^{<\NN}$. For our notion of largeness for well-founded trees, we use a property defined by Sipser in \cite{sipser1984}:

\begin{definition}

Let $B$ be a collection of well-founded trees and $w \in \NN^{<\NN}$. Given a well-founded tree $T$, we define $wT = \{wx : x \in T\}$. We call $B$ \emph{large at $w$} if for every well-founded tree $T$, we have $S \cap [w] = wT$ for some tree $S \in B$.

\end{definition}

\begin{lemma}
\label{wf-largeness}

If $B = \bigcup_{k \in \NN} B_k$ is large at $w$, then $B_k$ is large at $wk$ for some $k \in \NN$.

\end{lemma}

\begin{proof}

Suppose not. Then for every $k \in \NN$, there is a well-founded tree $T_k$ such that $S \cap [wk] \neq wkT_k$ for all $S \in B_k$. The tree $T = \{\emptyset\} \cup \bigcup_{k \in \NN} kT_k$ is well-founded. If we had $S \cap [w] = wT$ for any $S \in B$, then we would have $S \in B_k$ for some $k$ and $S \cap [wk] = wkT$, contradicting the construction of $T_k$. Thus $S \cap [w] \neq wT$ for any $S \in B$. The fact that $B$ is large at $w$ then gives us our desired contradiction.
\end{proof}

Note also that if $B$ is large at $w$, then $B$ is also large at $wz$ for all $z \in \NN^{<\NN}$. With our two notions of largeness in hand, we are now ready to prove the main claim. We identify trees with elements of $2^\NN$ by fixing a bijection $\varphi \colon \NN \rightarrow \NN^{<\NN}$ and mapping a tree $T$ to the string $x \in 2^\NN$ with $x_i = 1 \iff \varphi(i) \in T$.

\begin{theorem}
\label{ill-founded-trees}

The set of ill-founded trees is not Borel.

\end{theorem}

\begin{proof}

Let $A$ denote the set of ill-founded trees and $B$ denote the set of well-founded trees. Because $A$ is monotone, it suffices by Theorem~\ref{monotone-bcg} to exhibit a winning strategy in $\mBCG_{A,B}$ for Player II. We claim that Player II can win by always maintaining the invariant that, for all $n \in \NN$, there is a string $w^{(n)} \in \NN^{<\NN}$ such that $X_{A^{(n)}}$ is everywhere nonmeager in $[w^{(n)}]$ and $B^{(n)}$ is large at $w^{(n)}$.

The invariant holds initially since $X_A = \NN^\NN$ is everywhere nonmeager and $B$ is large at the empty string. To see that Player II can maintain the invariant, suppose $X_{A^{(n)}}$ is everywhere nonmeager in $[w^{(n)}]$ and $B^{(n)}$ is large at $w^{(n)}$ for some string $w^{(n)} \in \NN^{<\NN}$. If $A^{(n)} = \bigcup_{j \in \NN} A^{(n)}_j$, then Lemma~\ref{everywhere-nonmeager} tells us that some $X_{A^{(n)}_j}$ is everywhere nonmeager in $[w^{(n)}z]$ for some $z \in \NN^{<\NN}$. Thus Player II can choose $j$, and the invariant holds for $w^{(n+1)} = w^{(n)}z$. If, on the other hand, Player I partitions $B^{(n)} = \bigcup_{k \in \NN} B^{(n)}_k$, then Lemma~\ref{wf-largeness} tells us that some $B^{(n)}_k$ is large at $w^{(n)}k$. Thus Player II can choose $k$, and the invariant holds for $w^{(n+1)} = w^{(n)}k$.

To show that this strategy is winning for Player II, it suffices to prove that for any $n \in \NN$ and string $v \in \NN^{<\NN}$, we have either $v \notin S$ for some $S \in A^{(n)}$ or $v \in T$ for some $T \in B^{(n)}$. Indeed, if $v$ and $w^{(n)}$ are compatible, then the largeness of $B^{(n)}$ at $w^{(n)}$ gives us a tree $T \in B^{(n)}$ with $v \in T$. If $v$ and $w^{(n)}$ are incompatible, then we can choose $x \in X_{A^{(n)}}$ with $v$ and $x$ incompatible. The tree $S$ consisting precisely of the finite prefixes of $x$ satisfies $S \in A^{(n)}$ and $v \notin S$, as desired.
\end{proof}

\begin{remark}

There is a sense in which this argument is stronger than the usual proof by universal sets and diagonalization that ill-founded trees are not Borel. The standard proof tells us that Player II has some winning strategy in $\mBCG_{A,B}$, while our proof gives us a special kind of winning strategy. In particular, a generic winning strategy chooses each $j$ as a function of a partition $A^{(n)}_0, A^{(n)}_1, \dots$ of $A^{(n)}$ and set $B^{(n)}$. Similarly, a generic winning strategy chooses each $k$ as a function of a partition $B^{(n)}_0, B^{(n)}_1, \dots$ and set $A^{(n)}$. Our winning strategy, on the other hand, can be viewed as two distinct components: one which chooses $j$ based only on the partition of $A^{(n)}$ and one which chooses $k$ based only on the partition of $B^{(n)}$. These two components communicate with each other by reading and setting the string $w^{(n)} \in \NN^{<\NN}$. If we tried to frame a generic winning strategy in this manner, the two components would need to use an infinite object to communicate, rather than just a finite string.

\end{remark}

\section{Implications for finite circuits}
\label{sec:finite-circuits}

As the set of ill-founded trees is analytic (one can nondeterministically guess the branch), we have proven the following using KW games (in Theorem~\ref{analytic-separation} and Theorem~\ref{ill-founded-trees}):
\begin{enumerate}
    \item A set is Borel if and only if it is simultaneously analytic and co-analytic.
    \item The classes of Borel sets, analytic sets, and co-analytic sets are all distinct.
\end{enumerate}

These facts have corresponding statements for finite circuits. The first statement corresponds to the claim $\Pclass/\poly = \NP/\poly \cap \coNP/\poly$, and the second corresponds to the proposition that $\Pclass/\poly \neq \NP/\poly \neq \coNP/\poly$ (which implies $\Pclass \neq \NP \neq \coNP$). Adapting the arguments in this paper to the finite case provides a potential route to resolving the $\Pclass$ vs.\ $\NP$ and $\NP$ vs.\ $\coNP$ problems. However, there are many differences between the finite and infinite settings which suggest that such a task will be difficult:
\begin{itemize}
    \item The claim $\Pclass/\poly = \NP/\poly \cap \coNP/\poly$ is generally believed to be false (this would imply, for example, that the factoring problem has polynomial-sized circuits), while the infinite version of this claim is true.
    \item Monotone Borel functions have monotone circuits (Theorem~\ref{monotone-borel}), while monotone functions in $\Pclass$ do not all have polynomial-sized monotone circuit families \cite{tardos1988}.
    \item $D(f)$ and $\log_2(\chi(f))$ are quadratically related in finite communication complexity, while all infinite KW games (even those without protocols) have countable partitions (Theorem~\ref{kw-partition}).
\end{itemize}

It seems that well-foundedness, although related to the size of finite circuits, is not a perfect analog.

\section{Acknowledgments}

I would like to thank Anton Bernshteyn for providing invaluable guidance, insightful observations, and many thoughtful questions leading to new discoveries. He found clever ways to simplify the proofs of Theorem~\ref{analytic-separation} and Theorem~\ref{monotone-borel}. I would also like to thank Alexander Sherstov for an excellent introduction to communication complexity.

\printbibliography

\end{document}